\documentclass[a4paper,12pt]{article}

\usepackage{amsfonts}
\usepackage{amsthm}
\usepackage{amsmath}
\usepackage{amssymb}
\usepackage[all]{xy}
\usepackage[utf8]{inputenc}

\theoremstyle{plain}
\newtheorem{theorem}{Theorem}[section]
\newtheorem{corollary}[theorem]{Corollary}

\newtheorem{proposition}[theorem]{Proposition}

\theoremstyle{definition}
\newtheorem{example}[theorem]{Example}

\newcommand{\QQ}{{\mathbb Q}}
\newcommand{\RR}{{\mathbb R}}
\newcommand{\ZZ}{{\mathbb Z}}

\newcommand{\calS}{{\mathcal{S}}}
\newcommand{\calD}{{\mathcal{D}}}
\newcommand{\calA}{{\mathcal{A}}}
\newcommand{\calK}{{\mathcal{K}}}
\newcommand{\calG}{{\mathcal{G}}}
\newcommand{\calM}{{\mathcal{M}}}
\newcommand{\calF}{{\mathcal{F}}}
\newcommand{\calH}{{\mathcal{H}}}
\newcommand{\calC}{{\mathcal{C}}}
\newcommand{\calQ}{{\mathcal{Q}}}
\newcommand{\calW}{{\mathcal{W}}}

\newcommand{\Ho}{{\rm Ho}}
\newcommand{\map}{{\rm map}}
\newcommand{\Map}{{\rm Map}}

\newcommand{\calMo}{\calM^{\circ}}

\newcommand{\sSet}{{\rm sSet}}
\newcommand{\Spectra}{{\rm Sp}}

\SelectTips{cm}{}

\numberwithin{equation}{section}

\begin{document}

\title{Homotopy reflectivity is equivalent to the weak Vop\v{e}nka principle}
\author{Carles Casacuberta and Javier J. Guti\'errez}
\date{}
\maketitle

\begin{abstract}
Homotopical localizations with respect to (possibly proper) classes of maps are known to exist assuming the validity of a large-cardinal axiom from set theory called Vop\v{e}nka's principle. In this article, we prove that each of the following statements is equivalent to an axiom of lower consistency strength than Vop\v{e}nka's principle, known as weak Vop\v{e}nka's principle: (a)~Localization with respect to any class of maps exists in the homotopy category of simplicial sets; (b)~Localization with respect to any class of maps exists in the homotopy category of spectra; (c)~Localization with respect to any class of morphisms exists in any presentable $\infty$-category; (d)~Every full subcategory closed under products and fibres in a triangulated category with locally presentable models is reflective. Our results are established using Wilson's 2020 solution to a long-standing open problem concerning the relative consistency of weak Vop\v{e}nka's principle within the large-cardinal hierarchy.
\end{abstract}

\section*{Introduction}
\label{introduction}

One way to state the Vop\v{e}nka principle (VP) is that the class $\rm Ord$ of all the ordinals cannot be fully embedded into any locally presentable category.
In other words, there is no sequence of objects $\langle X_i \mid i\in{\rm Ord}\rangle$ indexed by all the ordinals in a locally presentable category $\mathcal C$ such that ${\mathcal C}(X_i,X_j)$ is a singleton if $i\le j$ and the empty set if $i>j$. This statement is a large-cardinal axiom that cannot be proved in~ZFC. As explained in \cite{AR}, it has many important consequences in category theory.

For example, the Vop\v{e}nka principle implies that full subcategories of locally presentable categories are reflective if they are closed under limits, and they are coreflective if they are closed under colimits. However, as shown in~\cite[\S\,6.D]{AR}, the reflectivity of limit-closed full subcategories is equivalent to the statement that there is no full embedding of ${\rm Ord}^{\rm op}$ into any locally presentable category, which is a weaker condition than Vop\v{e}nka's principle. This is called \emph{weak Vop\v{e}nka principle} (WVP). It can be rephrased by saying that there is no sequence of objects $\langle X_i \mid i\in{\rm Ord}\rangle$ indexed by all the ordinals in a locally presentable category $\mathcal C$ such that ${\mathcal C}(X_i,X_j)$ is a singleton if $i\ge j$ and the empty set if $i<j$. 

In \cite[Definition~I.11]{AR1}, Ad\'amek and Rosick\'y introduced another variation, which they named \emph{semi-weak Vop\v{e}nka principle} (SWVP), by replacing the condition that ${\mathcal C}(X_i,X_j)$ be a singleton if $i\ge j$ by the condition that ${\mathcal C}(X_i,X_j)$ be nonempty if $i\ge j$. As shown in~\cite{AR1}, SWVP is equivalent to the statement that every injectivity class in a locally presentable category is weakly reflective.

While it is easy to see that there are implications ${\rm VP}\Rightarrow{\rm SWVP}\Rightarrow {\rm WVP}$, whether the reverse implications are true or not remained as an open problem until Wilson proved in \cite{Wilson1} that WVP is equivalent to SWVP but not to VP. In fact the weak Vop\v{e}nka principle turned out to be equivalent to the statement that $\rm Ord$ is Woodin ---or, equivalently, that for every class $C$ there exists a $C$-strong cardinal; see~\cite{Wilson2}. This determines the precise place of the weak Vop\v{e}nka principle in the large-cardinal hierarchy, and shows that it is very far below Vop\v{e}nka's principle, even far below the existence of supercompact cardinals.

In another direction, it was shown in \cite{CSS} that the statement that all homotopical localizations of simplicial sets are $f$-localizations for some map $f$ is implied by Vop\v{e}nka's principle, and Prze\'zdziecki proved in \cite{Prz1} that the converse is true, that is, the statement that every homotopical localization of simplicial sets is an $f$\nobreakdash-localization for a single map $f$ is equivalent to Vop\v{e}nka's principle. Prze\'zdziecki also proved in \cite{Prz1} that the weak Vop\v{e}nka principle is equivalent to the statement that every orthogonality class in the category of groups is reflective, and he later proved in \cite{Prz2} that the same claim is true for the category of abelian groups.

In this article, we first prove that SWVP implies that localization with respect to any class of maps exists in the homotopy category of (pointed or unpointed) simplicial sets, and next we infer using \cite{Prz1} that the latter claim implies~WVP. Therefore, since SWVP and WVP are known to be equivalent by~\cite{Wilson1}, the claim that every (simplicially enriched) orthogonality class in simplicial sets is associated with a homotopical localization turns out to be equivalent to~WVP. We also prove that the same result holds for spectra. In fact, we show that WVP is equivalent to the claim that every full subcategory closed under products and fibres in any triangulated category with locally presentable models is reflective.

\newpage

As a consequence of our results, the existence of \emph{cohomological} localizations of simplicial sets or spectra is implied by the weak Vop\v{e}nka principle. This improves substantially a conclusion from~\cite{BCMR}, where a similar result was obtained assuming the existence of a proper class of supercompact cardinals.
Furthermore, using results from Bagaria and Wilson~\cite{BW}, we prove that the large-cardinal strength of the claim that cohomological localizations exist is strictly lower than~WVP.

Although the core of this article deals with simplicial sets and spectra, we formulate our main conclusions in the language of $\infty$-categories as well. The article is written so that both those who are familiar with $\infty$-categories and those who are less familiar can read it.
We use fibrant simplicial categories \cite{Bergner} as models for $\infty$-categories. For every simplicial model category $\calM$, the full subcategory $\calMo$ spanned by the objects that are fibrant and cofibrant in $\calM$ is enriched in Kan complexes and therefore can be viewed as an $\infty$\nobreakdash-category. Indeed, its coherent nerve $N(\calMo)$ is a quasi-category in the sense of~\cite{Joyal,Lurie}. As shown in~\cite{Lurie}, every \emph{presentable} $\infty$-category is equivalent to $\calMo$ for some combinatorial simplicial model category~$\calM$.

A crucial advantage of implementing our results in the framework of $\infty$\nobreakdash-cat\-egories is that the reflections on the homotopy category of $\calM$ obtained by means of WVP (see Section~\ref{reflections}) can be promoted to $\infty$-categorical localizations on~$\calMo$, as also found by Lo Monaco in \cite{LoMonaco} within a detailed study of the use of Vop\v{e}nka principles in $\infty$-category theory. Thus we also conclude that WVP is equivalent to the claim that $\calS$-localization exists for every class of morphisms $\calS$ in any presentable $\infty$-category. 

One could then ask if WVP is also sufficient to infer the existence of left Bousfield localizations of left proper combinatorial model categories $\calM$ with respect to arbitrary classes of morphisms. This is known to be true assuming~VP, by \cite[Theorem~2.3]{RT} or \cite[Lemma~1.4]{CCh}.
In the last section of this article, we prove that WVP implies indeed the existence of left Bousfield localization with respect to every class of morphisms~$\calS$, yet with several shortcomings. In general, we can only guarantee the existence of a (right) \emph{semi-model} category structure on the localized category~$\calM_\calS$, with factorizations that are not necessarily functorial. Moreover, we need to assume that all objects in $\calM$ are cofibrant. 
By \cite[Corollary~1.2]{Dugger1},
every combinatorial model category is Quillen equivalent to a simplicial one in which all objects are cofibrant. However, expliciting the interaction between Quillen equivalences and Bousfield localizations would require a more focused analysis.

\newpage

\noindent
\textbf{Acknowledgements.} We are indebted to Ji\v{r}\'{\i} Rosick\'y for sharing with us the main arguments used in Section~\ref{reflections}. We also thank Giulio Lo Monaco, Giuseppe Leoncini, and Leonid Positselski for their very useful remarks on earlier drafts of this article, as well as Joan Bagaria and Jeff Bergfalk for checking and improving the content of Section~\ref{cohomological-localizations}. This work was supported by MCIN/AEI under I+D+i grant PID2020-117971GB-C22 and by the Departament de Recerca i Universitats de 
la Generalitat de Catalunya with reference 2021 SGR 00697.

\section{Preliminaries}
\label{preliminaries}

A category $\calC$ is \emph{locally presentable} if it is cocomplete and there is a regular cardinal $\lambda$ and a set $\calA$ of $\lambda$\nobreakdash-presentable objects of $\calC$ such that every object of $\calC$ is a 
$\lambda$\nobreakdash-filtered colimit of objects from~$\calA$.
An object $X$ of $\calC$ is \emph{$\lambda$-presentable} if the functor $\calC(X,-)$ preserves $\lambda$\nobreakdash-filtered colimits; see \cite[\S 6.1]{GU} for further information about presentability.

By \cite[Corollary~1.28]{AR}, locally presentable categories are also complete.
Every category of structures in any first-order (possibly infinitary) language is locally presentable \cite[Chapter~5]{AR}, and the forgetful functor from structures into sets creates limits and colimits. 

\subsection{Large-cardinal principles}
\label{large-cardinals}

Our set-theoretic framework is that of the Zermelo--Fraenkel axioms with the axiom of choice (ZFC). By default, we work in the von Neumann universe $V$ of all sets, and a \emph{class} is any collection of~sets $\{x\mid\varphi(x)\}$ defined by some formula $\varphi$ in the first-order language of set theory with a single free variable, possibly with parameters. We denote by Ord the class of all ordinals.

Our results can also be interpreted in a chosen Grothendieck universe. This formalism is equivalent to assuming the existence of an inaccessible cardinal $\kappa$ and considering only sets that belong to the stage $V_\kappa$ of the cumulative hierarchy \cite{Jech}, i.e., sets whose rank is smaller than~$\kappa$. Such sets are called \emph{small sets} when necessary to avoid confusion, while the subsets of $V_\kappa$ are \emph{classes}. A~class $\calC$ is defined by a formula $\varphi$ if $\calC=\{x\in V_\kappa \mid V_\kappa \models\varphi(x)\}$. If this formalism is adopted, then Ord denotes the class of ordinals $i$ such that $i<\kappa$, and, consequently, large-cardinal statements such as Vop\v{e}nka's principle are understood to hold in~$V_\kappa$.

Categories are assumed to be \emph{locally small}, meaning that, for every two objects $X$ and~$Y$, the morphisms $X\to Y$ form a set ---or a small set, if one chooses to work in a Grothendieck universe.
We also assume that $\infty$\nobreakdash-categories are locally small, that is, for all objects $X$ and~$Y$, the mapping space $\map(X,Y)$ is weakly equivalent to a (small) Kan simplicial set. 

The class of ordinals ${\rm Ord}$ is viewed as a category with a unique morphism $\lambda\to\mu$ if and only if $\lambda\le\mu$.
The main large-cardinal principles considered in this article can be stated as follows.
\begin{itemize}
\item[(a)] Vop\v{e}nka's principle (VP):
There is no sequence $\langle X_i \mid i\in{\rm Ord}\rangle$ of objects in a locally presentable category $\mathcal C$ such that ${\mathcal C}(X_i,X_j)$ is a singleton if $i\le j$ and it is empty if $i>j$. 
\item[(b)] Weak Vop\v{e}nka's principle (WVP):
There is no sequence $\langle X_i \mid i\in{\rm Ord}\rangle$ of objects in a locally presentable category $\mathcal C$ such that ${\mathcal C}(X_i,X_j)$ is a singleton if $i\ge j$ and it is empty if $i<j$. 
\item[(c)] Semi-weak Vop\v{e}nka's principle (SWVP):
There is no sequence of objects $\langle X_i \mid i\in{\rm Ord}\rangle$ in a locally presentable category $\mathcal C$ such that
${\mathcal C}(X_i,X_j)$ is nonempty if $i\ge j$ and it is empty if $i<j$. 
\end{itemize}

As shown in~\cite{Wilson1,Wilson2}, WVP is equivalent to SWVP, and also equivalent to the statement that the class of all ordinals is Woodin. It is also proved in \cite{Wilson1} that the large-cardinal strength of VP is much higher.

\subsection{Orthogonality}
\label{orthogonality}

A morphism $f\colon A\to B$ and an object $X$ in a category $\calC$ are called \emph{orthogonal} if the function $\calC(B,X)\to\calC(A,X)$ sending every $g\colon B\to X$ to $g\circ f$ is a bijection of sets. This notion goes back to \cite{FK} and is central in most instances of localization theory. The dual notion, where one asks the function $\calC(X,A)\to\calC(X,B)$ sending every $h\colon X\to A$ to $f\circ h$ to be bijective, will not be considered in this article.

A morphism $r\colon X\to D$ in a category $\calC$ is called a \emph{reflection} of $X$ onto a class $\calD$ of objects if $D\in\calD$ and $r$ is orthogonal to all objects in~$\calD$, that is, for every morphism $f\colon X\to E$ where $E\in\calD$ there is a unique $g\colon D\to E$ such that $g\circ r=f$. If the uniqueness condition on $g$ is dropped, then $r$ is called a \emph{weak reflection} of $X$ onto~$\calD$.

Homotopical versions of orthogonality can be defined using either the formalism of Quillen model categories or in terms of $\infty$-categories. Although the two approaches are essentially equivalent, we provide details for both and describe their relationship in the next subsections.

\subsubsection{Orthogonality in simplicial model categories}
\label{model-orthogonality}

Let $\calM$ be a simplicial model category, and denote the simplicial enrichment by~$\Map(-,-)$.
In this article we assume that model categories admit functorial factorizations, except in some parts of Section~\ref{functoriality}. Thus, we may choose a
cofibrant replacement functor $Q$ and a fibrant replacement functor $R$ in~$\calM$,
equipped with natural transformations $Q\to{\rm Id}$ and ${\rm Id}\to R$. 

We say that a morphism $f\colon A\to B$ and an object $X$ in $\calM$ are \emph{simplicially orthogonal} if the map of simplicial sets
\[
\Map(QB,RX)\longrightarrow\Map(QA,RX)
\]
induced by $Qf\colon QA\to QB$
is a weak equivalence. We denote this fact by writing $f\perp X$, and call them \emph{orthogonal} instead of simplicially orthogonal if no confusion can arise.

The objects of the \emph{homotopy category} $\Ho(\calM)$ are those of $\calM$, and the set of morphisms from $X$ to $Y$ in $\Ho(\calM)$ is the set of connected components $\pi_0(\Map(QX,RY))$, which is denoted by $[X,Y]$ and coincides with the set of homotopy classes of morphisms $QX\to RY$. 
If a morphism $f\colon A\to B$ and an object $X$ are simplicially orthogonal in $\calM$ then they are orthogonal in~$\Ho(\calM)$, although the converse is not true in general.

For a class $\calS$ of morphisms in~$\calM$, its \emph{orthogonal} class of objects is 
\[
\calS^{\perp}=
\{X\in\calM \mid f\perp X \;\text{for all $f\colon A\to B$ in $\calS$}\},
\]
and, similarly, given a class $\cal D$ of objects, we denote
\[
\calD^{\perp}=\{f\colon A\to B\mid f\perp X \;\text{for all $X$ in $\calD$}\}.
\]

Given a class of morphisms $\calS$ in~$\calM$, an \emph{$\calS$-localization} of an object $X$ is a morphism $\ell\colon X\to LX$ in $\calS^{\perp\perp}$ such that $LX\in\calS^{\perp}$. The objects in $\calS^{\perp}$ are called \emph{$\calS$-local} and the morphisms in $\calS^{\perp\perp}$ are called \emph{$\calS$\nobreakdash-equiv\-alences}. 

If $\ell\colon X\to LX$ is an $\calS$-localization of~$X$, then $\ell$ is also a reflection of $X$ onto $\calS^\perp$ in the homotopy category $\Ho(\calM)$.
A~refinement of this statement is given in Proposition~\ref{Prop1-2} below.

If every object of $\calM$ admits an $\calS$-localization, then the full subcategory spanned by $\calS^\perp$ in $\Ho(\calM)$ is \emph{reflective}, which means that its inclusion into $\Ho(\calM)$ has a left adjoint. We follow standard terminology by calling the corresponding reflector an \emph{$\calS$-localization} functor on~$\Ho(\calM)$, in spite of the ambiguity of this term. In general, localizations need not be functorial on the model category~$\calM$.

\subsubsection{Orthogonality in \boldmath$\infty$-categories}
\label{infinity-orthogonality}

Higher-categorical orthogonality is defined as follows. A~morphism $f\colon A\to B$ and an object $X$ in an $\infty$-category $\calC$ are \emph{orthogonal} if the induced map of 
$\infty$-groupoids
$\map(B,X)\to\map(A,X)$
is an equivalence. 

This notion can be made more explicit by choosing a specific model for $\infty$\nobreakdash-categories. In this article we use fibrant simplicial categories, i.e., categories enriched in simplicial sets that are fibrant in the Bergner model structure \cite{Bergner}, which means precisely that $\map(X,Y)$ is a Kan complex for all $X$ and~$Y$. If $\calC$ is a fibrant simplicial category, then its coherent nerve $N\calC$ is a quasi-category, that is, a fibrant object in the Joyal model structure on simplicial sets~\cite{Joyal}. 

The \emph{homotopy category} ${\rm h}\calC$ of an $\infty$-category $\calC$ has the same objects as~$\calC$ and ${\rm h}\calC(X,Y)=\pi_0(\map(X,Y))$, which we also denote by $[X,Y]$.

Our terminology and notation for orthogonality is the same as in simplicial model categories. Thus, if $\calS$ is a class of morphisms in an $\infty$-category~$\calC$, an \emph{$\calS$-localization} of an object $X$ is a morphism $\ell\colon X\to LX$ in $\calS^{\perp\perp}$ such that $LX\in\calS^\perp$.
If every $X$ has an $\calS$-localization, then the full subcategory spanned by~$\calS^{\perp}$ is called \emph{reflective},
as in \cite[5.2.7.9]{Lurie}, and $L$ becomes an endo\-functor of~$\calC$, called a \emph{reflector}. Its codomain restriction is left adjoint, in the $\infty$-categorical sense, to the inclusion of $\calS^{\perp}$ into~$\calC$.

As in the previous subsection, every $\calS$-localization $\ell\colon X\to LX$ in an $\infty$-category $\calC$ yields a reflection onto $\calS^{\perp}$ in the homotopy category ${\rm h}\calC$. 

\subsubsection{Comparison}
\label{comparison}

If $\calM$ is a simplicial model category with simplicial enrichment $\Map(-,-)$, then its full subcategory $\calMo$ of fibrant-cofibrant objects is a fibrant simplicial category with $\map(X,Y)=\Map(X,Y)$, and hence it can be viewed as an $\infty$-category such that ${\rm h}(\calMo)$ is equivalent to~$\Ho(\calM)$ through the functor $RQ\colon \calM\to\calMo$, where $Q$ is a cofibrant replacement functor and $R$ is a fibrant replacement functor on~$\calM$.
Furthermore, orthogonality $f\perp X$ in $\calM$ is equivalent to orthogonality $RQf\perp RQX$ in $\calMo$, since \[
\map(RQB,RQX)=\Map(RQB,RQX)\simeq\Map(QB,RX),
\]
by the homotopy invariance of $\Map(-,-)$ if the domain is cofibrant and the codomain is fibrant.

If $\calS$ is a class of morphisms in $\calMo$, then the orthogonal complements of $\calS$ in $\calM$ and $\calMo$, which we denote respectively with
$\calS^\perp(\calM)$ and $\calS^\perp(\calMo)$, are related as follows. We denote by $RQ\calS$ the class $\{RQf\mid f\in\calS\}$.

\begin{proposition}
\label{Prop1-1}
Let $\calM$ be a simplicial model category with cofibrant replacement $Q$ and fibrant replacement~$R$. Given a class of morphisms $\calS$ in $\calM$, the class $\calS^\perp(\calM)$ is the closure of $(RQ\calS)^\perp(\calMo)$ under weak equivalences.
\end{proposition}

\begin{proof}
The class $\calS^\perp(\calM)$ is closed under weak equivalences and coincides with $(RQ\calS)^\perp(\calM)$, since $RQf$ is weakly equivalent to~$f$ for all~$f$.
Clearly $(RQ\calS)^\perp(\calMo)\subseteq (RQ\calS)^\perp(\calM)$. If $X\in\calS^\perp(\calM)$ then $RQX\in(RQ\calS)^\perp(\calMo)$; hence $X$ is weakly equivalent to an object in~$(RQ\calS)^\perp(\calMo)$.
\end{proof}

The class $\calS^\perp(\calM)$ is closed under homotopy limits in~$\calM$, while the class $(RQ\calS)^\perp(\calMo)$ is closed under $\infty$-categorical limits in~$\calMo$. 
Although $\calMo$ need not be closed under products in~$\calM$, the inclusion $\calMo\hookrightarrow\calM$ preserves products up to weak equivalence.
Specifically, if $\{X_i\}_{i\in I}$ is a (small) set of objects in~$\calMo$, then their product in $\calMo$ is a cofibrant replacement of $\prod_{i\in I}X_i$ in~$\calM$.

\subsubsection{Underlying \boldmath$\infty$-category}
\label{underlying}

Simplicial orthogonality can be defined in any model category $\calM$ (not necessarily simplicial) by means of the \emph{hammock localization} \cite{DK1,DK2} of $\calM$ with respect to the class of weak equivalences of~$\calM$. This is a simplicial category $L^H\calM$ with the same objects as~$\calM$, whose homotopy category is equivalent to $\Ho(\calM)$. Hence, mapping spaces in $L^H\calM$ allow us to define orthogonality and $\calS$-localizations in~$\calM$. 

If the model category $\calM$ is simplicial, then $L^H\calM$ and $\calMo$ are weakly equivalent as simplicial categories by \cite[\S\;4]{DK2}, and therefore they define the same orthogonality relation.

If $\mathbf{R}$ is a fibrant replacement functor for the Bergner model structure on simplicial categories, then $\mathbf{R} L^H\calM$ is called the \emph{underlying} $\infty$-category of $\calM$, as in \cite{Hinich,Mazel-gee}. Indeed, when the model category $\calM$ is simplicial, $\mathbf{R} L^H\calM$ and $\calMo$ become equivalent as $\infty$-categories, since they are weakly equivalent as simplicial categories.

\subsubsection{Homotopy reflections}
\label{homotopy-reflections}

In this subsection we show that, under mild assumptions, an $\calS$-localization in a simplicial model category $\calM$ is the same as a reflection onto $\calS^\perp$ in $\Ho(\calM)$, and similarly for an $\infty$-category~$\calC$. We note that, if $\calC=\calMo$ for a simplicial model category $\calM$, then the tensoring and cotensoring of $\calM$ induce a tensoring and cotensoring of~$\calC$. 

\begin{proposition}
\label{Prop1-2}
Let $\calS$ be a class of morphisms in a category $\calC$ enriched in Kan complexes and cotensored over simplicial sets, and let $X$ be any object of~$\calC$. A morphism $r\colon X\to D$ is an $\calS$-localization of~$X$ if and only if $r$ is a reflection of $X$ onto $\calS^{\perp}$ in the homotopy category ${\rm h}\calC$.
\end{proposition}

\begin{proof}
Suppose first that $r\colon X\to D$ is an $\calS$-localization. Then for every $E\in\calS^\perp$ the morphism $r$ induces a weak equivalence
\[
\map(D,E)\simeq \map(X,E),
\]
which yields a bijection $[D,E]\cong [X,E]$ on $\pi_0$. This means precisely that $r$ is a reflection onto~$\calS^\perp$ in~${\rm h}\calC$.
For the converse, we let $r\colon X\to D$ be a reflection onto~$\calS^\perp$, and aim to prove that $r\in\calS^{\perp\perp}$.
Hence we need to prove that for every $E\in\calS^\perp$ the induced map
\begin{equation}
\label{r*}
r^*\colon \map(D,E)\longrightarrow \map(X,E)
\end{equation}
is a weak equivalence.
For this, we prove that \eqref{r*} induces bijections
\[
[W,\map(D,E)]\cong [W,\map(X,E)]
\]
for every simplicial set~$W$.
Using the cotensoring of $\calC$, we obtain bijections
\begin{align}
\label{wmapde}
[W,\map(D,E)] & \cong 
\pi_0\big(\map(D,E)^W\big)
\cong \pi_0\big(\map\big(D,E^W\big)\big)\cong\big[D,E^W\big].
\end{align}
Next we show that $E^W\in\calS^\perp$.
If $f\colon A\to B$ is in~$\calS$, then, since $E\in\calS^\perp$, the morphism $f$ induces a weak equivalence $\map(B,E)\simeq\map(A,E)$. Therefore, for every simplicial set $V$ we have
\begin{align*}
\big[V, {\rm map} & \big(B,E^W\big)\big] \cong
\pi_0\Big(\map(B,E^W\big)^V\Big) \cong
\pi_0\big(\map\big(B,E^{V\times W}\big)\big) \\[0.1cm] & \cong
[V\times W,{\rm map}(B,E)] 
\cong [V\times W,{\rm map}(A,E)] \cong \big[V,{\rm map}\big(A,E^W\big)\big].
\end{align*}
Hence $f$ induces a weak 
equivalence ${\rm map}\big(B,E^W\big)\simeq {\rm map}\big(A,E^W\big)$, as needed. Going back to~\eqref{wmapde}, we conclude that $r$ induces bijections
\[
[W,\map(D,E)]\cong\big[D,E^W\big]\cong\big[X,E^W\big]\cong[W,\map(X,E)],
\]
which completes the argument.
\end{proof}

\begin{corollary}
\label{Cor1-3}
Let $\calS$ be a class of morphisms in a simplicial model category~$\calM$, and let $X$ be an object of~$\calM$. A morphism $r\colon X\to D$ in $\calM$ is an $\calS$\nobreakdash-loc\-al\-ization of $X$ if and only if $r$ is a reflection of $X$ onto $\calS^{\perp}$ in the homotopy category $\Ho(\calM)$.
\end{corollary}

\begin{proof}
Only one implication needs to be proved. Suppose that $r\colon X\to D$ is a reflection onto $\calS^\perp$ in $\Ho(\calM)$. Let $E\in\calS^\perp$, which we may assume fibrant. Then the induced map
\[
r^*\colon \Map(QD,E)\longrightarrow\Map(QX,E)
\]
is shown to be a weak equivalence with the same argument used for \eqref{r*}
in the proof of Proposition~\ref{Prop1-2}. Cotensoring with a simplicial set is a right Quillen endofunctor of $\calM$ and therefore it preserves fibrant objects.
\end{proof}

The assumption that a morphism $X\to D$ is given is crucial in Corollary~\ref{Cor1-3}.
Otherwise, a reflection of $X$ onto $\calS^\perp$ in $\Ho(\calM)$ need not lift to an $\calS$-localization of~$X$, unless $X$ is cofibrant in~$\calM$.

We also emphasize that it is not true that every localization functor $L$ on $\Ho(\calM)$ ---onto an arbitrary class~$\calD$--- is an $\calS$-localization for some class of morphisms~$\calS$. A~counter\-example is shown in Example~\ref{ex3} below. Typically such a failure is due to the fact that the class $\calD$ does not share the closure properties of a simplicial orthogonal complement of a class of morphisms. Yet, the corresponding $\infty$\nobreakdash-categorical statement holds, as stated next.

\begin{proposition}
\label{Prop1-4}
Let $L$ be a reflector on an $\infty$-category $\calC$, and let $\calS$ be the class of morphisms $f$ such that $Lf$ is invertible. Then $L$ is an $\calS$-localization.
\end{proposition}

\begin{proof}
For each object~$X$, the reflection $\ell_X\colon X\to LX$ belongs to~$\calS$, hence also to~$\calS^{\perp\perp}$. Moreover, $LX\in\calS^\perp$, since if $f\colon A\to B$ is in $\calS$ then, similarly as in \cite[5.5.4.2]{Lurie}, the following diagram is homotopy commutative and three of its arrows are weak equivalences by assumption:
\[
\xymatrix@C=4pc@R=3pc{
  \map(LB,LX) \ar[r]^-{(Lf)^*}_{\simeq} \ar[d]_{\ell_B^*}^{\simeq} & \map(LA,LX) \ar[d]^{\ell_A^*}_{\simeq} \\
  \map(B,LX) \ar[r]^-{f^*} & \map(A,LX),
}
\]
from which it follows that $f^*$ is also a weak equivalence.
\end{proof}

\subsubsection{Orthogonality in spectra}
\label{spectra-orghogonality}

By a \emph{spectrum} we mean an object of any simplicial model category $\Spectra$ whose underlying category is locally presentable and whose homotopy category is equivalent to the classical stable homotopy category ---for example, the Bousfield--Friedlander category~\cite{BF} or the category of symmetric spectra over simplicial sets~\cite{HSS}.

Let $Q$ be a cofibrant replacement functor and $R$ a fibrant replacement functor on~$\Spectra$. Given spectra $X$ and $Y$, the function spectrum $F(X,Y)$ is defined by means of Brown representability in $\Ho(\Spectra)$ as a representing spectrum for the functor $[X\wedge(-),Y]$.
The homotopy groups of $F(X,Y)$ coincide in nonnegative degrees with those of the simplicial set $\Map(QX,RY)$, where $\Map(-,-)$ denotes the simplicial enrichment of~$\Spectra$. Hence, if we denote by $F^c(X,Y)$ the connective cover of $F(X,Y)$, then orthogonality of spectra in the sense of Subsection~\ref{model-orthogonality} can alternatively be formulated as follows: a map $f\colon A\to B$ and a spectrum $X$ are orthogonal if and only if the map
\[
F^c(B,X)\longrightarrow F^c(A,X)
\]
induced by $f$ is a weak equivalence of spectra, i.e., it induces isomorphisms of all homotopy groups; see \cite{Bousfield,CG} for more information.

If a simplicial model category $\calM$ is stable, then every class of the form $\calS^{\perp}$ is \emph{semicolocalizing}, that is, $\calS^{\perp}$ is closed under fibres, products, and extensions, hence also under retracts and desuspension \cite[\S\,1.2]{CGR}. Moreover, if $\calS$ is closed under desuspension (or, equivalently, $\calS^\perp$ is closed under suspension), then $\calS^{\perp}$ is \emph{colocalizing}, i.e., triangulated and closed under products.

For an arbitrary simplicial model category $\calM$, each class of the form $\calS^\perp$ is closed under homotopy limits and each class of the form $\calS^{\perp\perp}$ is closed under homotopy colimits; see~\cite{Hirschhorn} for a detailed proof. Moreover, both $\calS^\perp$ and $\calS^{\perp\perp}$ are closed under homotopy retracts.

\section{From weak reflections to reflections}
\label{reflections}

Recall from Section~\ref{orthogonality} that a full subcategory $\calD$ of a category $\calM$ is \emph{weakly reflective} if for every object $X\in\calM$ there is a morphism $f\colon X\to X^*$ with $X^*\in\calD$ such that, for every $g\colon X\to D$ with $D\in\calD$ there is an $h\colon X^*\to D$ (not necessarily unique) such that $h\circ f=g$. We then say that $f\colon X\to X^*$ is a \emph{weak reflection} of $X$ onto~$\calD$. If a weakly reflective subcategory is closed under retracts, then it is closed under products \cite[Remark~4.5(3)]{AR}.

If we denote by $(X\downarrow\calD)$ the coslice category of $X$ over~$\calD$, whose objects are morphisms $X\to D$ and whose morphisms are commutative triangles, then a weak reflection $X\to X^*$ onto $\calD$ is a weakly initial object of $(X\downarrow\calD)$.

Similarly, a subcategory $\calD$ of an $\infty$-category $\calC$ is \emph{weakly reflective} if the coslice category $(X\downarrow\calD)$, most often denoted by $\calD_{X/}$ in the context of $\infty$\nobreakdash-categories, has a weakly initial object for every object $X$ of~$\calC$. The $\infty$\nobreakdash-category $\calD_{X/}$ is defined as a pullback of the forgetful functor $\calC_{X/}\to\calC$ and the inclusion $\calD\hookrightarrow\calC$ over~$\calC$. An object $X\to D$ in $\calD_{X/}$ is weakly initial if the mapping space from it to any other object of $\calD_{X/}$ is nonempty.

\subsection{Existence of weak reflections}
\label{weak-reflections}

The following result was proved in \cite[Theorem~I.9]{AR1} 
for classes closed under products and retracts assuming  Vop\v{e}nka's principle, and it was pointed out in \cite[Remark~I.10]{AR1} that the semi-weak Vop\v{e}nka principle was sufficient for the validity of the proof. In our version, we use a similar argument to the one in~\cite{AR1},  without assuming closedness under retracts. 

\begin{proposition}
\label{Prop2-1}
Suppose that the semi-weak Vop\v{e}nka principle holds. If $\calM$ is any locally presentable category, then every full subcategory of $\calM$ closed under products is weakly reflective.
\end{proposition}

\begin{proof}
Suppose given a full subcategory $\calD$ of $\calM$ closed under products. Hence, in particular, $\calD$ contains the terminal object of~$\calM$, and we implicitly assume that $\calD$ is closed under isomorphisms. 

Let $V=\cup_{i\in{\rm Ord}}\, V_i$ denote the cumulative hierarchy \cite{Jech} of sets in~ZFC. For each ordinal~$i$, let $\calD_i=\calD\cap V_i$ be the set of all objects in $\calD$ whose rank is smaller than~$i$, and let $\overline{\calD}_i$ be the closure of $\calD_i$ under products and isomorphisms. Hence $\calD_i\subseteq\calD_j$ if $i\le j$, and
\[ 
\calD=\bigcup_{i\in{\rm Ord}} \calD_i=\bigcup_{i\in{\rm Ord}} \overline{\calD}_i.
\]

For each object $X$ of~$\calM$ and every ordinal~$i$, let 
$\calF_i^X=(X\downarrow\calD_i)$, and let $r_i^X\colon X\to X_i$ be the product of all the objects of $\calF_i^X$ 
(in case that $\calF_i^X=\emptyset$, we let $X_i$ be the terminal object of~$\calM$). Let us write $r_i$ instead of~$r_i^X$.

Every object $Y\in\overline{\calD}_i$ is isomorphic to a product $\prod_{\lambda\in\Lambda}D_\lambda$ with $D_\lambda\in\calD_i$ (not necessarily distinct), and each morphism $f\colon X\to Y$ is determined by a collection of morphisms $\{ \delta_\lambda\colon X\to D_\lambda\}_{\lambda\in\Lambda}$. 
Here each $\delta_\lambda$ is in $\calF_i^X$ and therefore it can be factored as
\[
X\stackrel{r_i}{\longrightarrow} X_i\stackrel{p_\lambda}{\longrightarrow} D_\lambda,
\]
where the second arrow is a projection. The morphisms $p_\lambda$ (possibly repeated) jointly yield a morphism $f_i\colon X_i\to Y$ such that $f_i\circ r_i=f$. Hence $r_i\colon X\to X_i$ is a weak reflection of $X$ onto~$\overline{\calD}_i$. 

Moreover, since $\calF_i^X$ embeds as a subcategory of $\calF_j^X$ if $i\le j$, there is a projection $p_{ji}\colon X_j\to X_i$ such that $p_{ji}\circ r_j=r_i$ whenever $i\le j$.
Therefore, $\langle r_i\colon X\to X_i\mid i\in{\rm Ord}\rangle$ can be viewed as a sequence of objects in the coslice category $(X\downarrow\calM)$ equipped with morphisms $p_{ji}\colon r_j\to r_i$ if $i\le j$.

To obtain a weak reflection of $X$ onto~$\calD$, it suffices to find an ordinal $i_X$ such that, for all $j> i_X$, the morphism $r_j$ can be factorized as $r_j=q\circ r_{i_X}$ for some $q\colon X_{i_X}\to X_j$. 
Suppose the contrary. Then there exist ordinals $i_0 < i_1 < \cdots < i_s < \cdots$,
where $s$ ranges over all the ordinals, such that, if $s<t$, then
\begin{equation}
\label{nomaps}
(X\downarrow\calM)(r_{i_s},r_{i_t})=\emptyset.
\end{equation}
Since $\calM$ is locally presentable, $(X\downarrow\calM)$ is also locally presentable by \cite[Proposition~1.57]{AR}, and therefore \eqref{nomaps} is incompatible with the semi-weak Vop\v{e}nka principle, as $(X\downarrow\calM)(r_{i_s},r_{i_t})\ne\emptyset$ if $s\ge t$.
\end{proof}

From a weak reflection onto a full subcategory $\calD$ it is possible to construct a reflection onto $\calD$ under suitable assumptions ---our reference is \cite[Theorem~2.2]{CGR}.
First, we need that idempotents split in the category under study, that is, for every morphism $e\colon X\to X$ such that $e\circ e=e$ there are morphisms $f\colon X\to Y$ and $g\colon Y\to X$ for some $Y$ such that $e=g\circ f$ and $f\circ g={\rm id}$. Second, $\calD$ should be closed under retracts. Third and most fundamentally, $\calD$ should have the following closure property: every pair of parallel arrows between two objects in $\calD$ admits a weak equalizer in~$\calD$.
This happens in homotopy categories of model categories for classes of objects $\calD$ closed under homotopy limits, which is the case, for example, if $\calD=\calS^{\perp}$ for some class of morphisms~$\calS$. 

We state the following result in terms of model categories. As explained in Subsection~\ref{underlying}, orthogonality in a model category $\calM$ is meant through the mapping spaces of the hammock localization $L^H\calM$.

\begin{proposition}
\label{Prop2-2}
Let $\calM$ be a model category such that idempotents split in the homotopy category $\Ho(\calM)$. If $\calD$ is a weakly reflective full subcategory of $\Ho(\calM)$ closed under homotopy equalizers and retracts, then $\calD$ is reflective.
\end{proposition}

\begin{proof}
Let $X$ be any object of $\calM$ and assume it cofibrant without loss of generality. 
Let $r\colon X\to D$ be a weak reflection of $X$ onto~$\calD$ in~$\Ho(\calM)$, and assume that $D$ is fibrant and cofibrant.

Consider the set $\calK$ of all pairs $(\alpha_k,\beta_k)$ of morphisms in $\calM$ from $D$ to itself such that $\alpha_k\circ r\simeq\beta_k\circ r$. Let $u\colon E\to D$ be a homotopy equalizer in $\calM$ of the two morphisms $D\to\prod_{k}D$ given by $\prod_{k}\alpha_k$ and $\prod_{k}\beta_k$ respectively, with $E$ fibrant and cofibrant. Then $E\in\calD$ by assumption. 

Since $u\colon E\to D$ is a homotopy equalizer, there is a morphism $s\colon X\to E$ such that $u\circ s\simeq r$. As $E\in\calD$ and $r$ is a weak reflection of $X$ onto~$\calD$, there is a morphism $t\colon D\to E$ such that $t\circ r\simeq s$. Then the pair $(u\circ t,{\rm id})$ is in~$\calK$. Consequently, $u\circ t\circ u\simeq u$, and this implies that $t\circ u$ is idempotent in $\Ho(\calM)$. Hence $t\circ u$ splits, so there are morphisms $t'\colon E\to Z$ and $u'\colon Z\to E$ with $Z$ fibrant and cofibrant such that $u'\circ t'\simeq t\circ u$ and $t'\circ u'\simeq {\rm id}$. It then follows that $Z\in\calD$ since $Z$ is a retract of~$E$.
Then the morphism $t'\circ s$ from $X$ to $Z$ is a reflection of $X$ onto $\calD$ in $\Ho(\calM)$, as shown in detail within the proof of \cite[Theorem~2.2]{CGR}. 
\end{proof}

\begin{corollary}
\label{Cor2-3}
Suppose that the semi-weak Vop\v{e}nka principle holds. Let $\calM$ be a model category whose underlying category is locally presentable and such that idempotents split in the homotopy category $\Ho(\calM)$. Then $\calS^\perp$ is reflective in $\Ho(\calM)$ for every class $\calS$ of morphisms in~$\calM$.
\end{corollary}

\begin{proof}
Since $\calS^\perp$ is closed under products, Proposition~\ref{Prop2-1} tells us that $\calS^\perp$ is weakly reflective in~$\calM$.
Consequently, $\calS^\perp$ is also weakly reflective in $\Ho(\calM)$, and Proposition~\ref{Prop2-2} then implies that $\calS^\perp$ is reflective.
\end{proof}

For the validity of Corollary~\ref{Cor2-3}, it is not necessary to assume that $\calM$ be cofibrantly generated nor left proper ---these are standard assumptions on a model category for the purpose of constructing $\calS$\nobreakdash-localizations by means of the small object argument when $\calS$ is a (small) set.

On the other hand, $\calS$-localizations obtained using the small object argument are functorial on~$\calM$, while Corollary~\ref{Cor2-3} yields a reflector on~$\Ho(\calM)$, which, a priori, need not lift to an endofunctor of~$\calM$. 
We continue this discussion in Section~\ref{functoriality}.

\subsection{Spectra and simplicial sets}
\label{spectra-simplicial-sets}

The following is a special case of Corollary~\ref{Cor2-3}.

\begin{theorem}
\label{Thm2-4}
If the semi-weak Vop\v{e}nka principle holds, then $\calS$-localization exists in the homotopy category of spectra for every class of maps~$\calS$. 
\end{theorem}

\begin{proof}
Corollary~\ref{Cor2-3} applies to the category of Bousfield--Friedlander spectra \cite{BF} or to the category of symmetric spectra \cite{HSS}, since they are locally presentable and idempotents split in their homotopy category. In fact, idempotents split in any triangulated category with countable products~\cite{Neeman}.
\end{proof}

As for simplicial sets, it is neither true that idempotents split in the homotopy category $\Ho(\sSet)$ nor in the pointed homotopy category $\Ho(\sSet_*)$, but they do for pointed \emph{connected} simplicial sets~\cite{FH}. The proof of Proposition~\ref{Prop2-2} can be amended to deal with this circumstance. We provide details in order to highlight the necessary changes.

Let $\calS$ be a class of maps between simplicial sets.
As shown in \cite{Tai}, if some map in $\calS$ does not induce a bijection on $\pi_0$ then $\calS$-local spaces are contractible. Hence we may assume that all maps in $\calS$ induce bijections of connected components. 
If $f\colon A\to B$ is any such map and we denote by $\{f_{i}\colon A_{i}\to B_{i}\}_{i\in I}$ the collection of its restrictions to connected components of $A$ and~$B$, then a space $X$ is $f$-local if and only if $X$ is $f_{i}$-local for all~$i$.
Therefore we may also assume that the class $\calS$ consists of maps between connected simplicial sets, by replacing each map in the given class by the collection of its restrictions to connected components. 

We also use the fact that, if we choose basepoints so that each map in $\calS$ is basepoint-preserving, then the class of connected $\calS$-local spaces in the pointed category $\Ho(\sSet_*)$ is the same as the corresponding class in the unpointed category $\Ho(\sSet)$ by forgetting the basepoints. This follows, as observed in \cite[A.1]{Farjoun}, from the fact that for all pointed connected simplicial sets $A$ and $Y$ there is a fibration
\[
{\rm map}_*(A,Y)\longrightarrow {\rm map}(A,Y)\longrightarrow Y
\]
where $\map_*(A,Y)$ is the hom simplicial set in $\sSet_*$ and the right-hand arrow is evaluation at the basepoint of~$A$.

Consequently, for the proof of the next result we choose to work in the pointed category $\sSet_*$, and  there is no loss of generality if we restrict ourselves to connected spaces, since an $\calS$-localization of an arbitrary space is the disjoint union of the $\calS$-localizations of its connected components after choosing arbitrary basepoints in them.

\begin{theorem}
\label{Thm2-5}
If the semi-weak Vop\v{e}nka principle holds, then $\calS$-local\-iza\-tion exists in the homotopy categories of pointed or unpointed simplicial sets for every class of maps~$\calS$. 
\end{theorem}

\begin{proof}
By our previous remarks, it suffices to construct an $\calS$-localization in the pointed category $\sSet_*$ for every pointed \emph{connected} simplicial set~$X$, by assuming that maps in $\calS$ are basepoint-preserving maps between connected simplicial sets.
The proof of this theorem proceeds in the same way as the proof of Proposition~\ref{Prop2-2}. Let $\calD=\calS^{\perp}$. Given any~$X$, by Proposition~\ref{Prop2-1}, there is a weak reflection $r\colon X\to D$ of $X$ onto~$\calD$, since $\sSet_*$ is locally presentable. We may assume that $D$ is fibrant and view $r$ as a weak reflection in the homotopy category $\Ho(\sSet_*)$.
Since $X$ is connected, the image of $r$ is contained in the basepoint component $D_0$ of~$D$, and the codomain restriction $r_0\colon X\to D_0$ is still a weak reflection of $X$ onto~$\calD$ in $\Ho(\sSet_*)$.

Consider the set $\calK$ of all pairs $(\alpha_k,\beta_k)$ of maps from $D_0$ to itself such that $\alpha_k\circ f_0\simeq\beta_k\circ f_0$.
Let $u\colon E\to D_0$ be a homotopy equalizer in $\sSet_*$ of the two maps $D_0\to\prod_{k}D_0$ given by $\prod_{k}\alpha_k$ and $\prod_{k}\beta_k$ respectively, with $E$ fibrant. Then $E\in\calD$ since $\calD$ is closed under homotopy limits. Pick the basepoint component $E_0$ of~$E$, which is also in $\calD$ because it is a retract of~$E$.

Since $u\colon E\to D_0$ is a homotopy equalizer, there is a map $s\colon X\to E$ such that $u\circ s\simeq r_0$, and $s$ factors through a map $s_0\colon X\to E_0$ since $X$ is connected; that is, $s=i\circ s_0$, where $i\colon E_0\to E$ is the inclusion.
As $E_0\in\calD$ and $r_0$ is a weak reflection of $X$ onto~$\calD$, there is a map $t\colon D_0\to E_0$ such that $t\circ r_0\simeq s_0$. Then 
$(u\circ i\circ t,{\rm id})$ is in~$\calK$ and, consequently, $u\circ i\circ t\circ u\simeq u$.
This implies that $t\circ u\circ i$ is idempotent in $\Ho(\sSet_*)$. 

Now we use the fact that, according to \cite{FH}, idempotents split for pointed connected simplicial sets. Thus, since $E_0$ is connected, $t\circ u\circ i$ splits, so there are maps $t'\colon E_0\to Z$ and $u'\colon Z\to E_0$ with $Z$ fibrant such that $u'\circ t'\simeq t\circ u\circ i$ and $t'\circ u'\simeq {\rm id}$. It then follows that $Z\in\calD$ and $Z$ is connected, since $Z$ is a retract of~$E_0$.
The argument showing that the map $t'\circ s_0$ from $X$ to $Z$ is a reflection of $X$ onto $\calD$ in $\Ho(\sSet_*)$ is the same as in the proof of 
\cite[Theorem~2.2]{CGR}.
\end{proof}

\subsection{Reflectivity in \boldmath$\infty$-categories}
\label{infinity-categories}

In the higher-categorical context, it is not necessary to impose that idempotents split as in Proposition~\ref{Prop2-2}, since every $\infty$-category $\calC$ with products is idempotent complete. The reason is that, if $e\colon A\to A$ is idempotent, then a splitting $s\colon B\to A$ is an equalizer of $e$ and the identity. However, as observed in \cite[4.4.5.19]{Lurie}, an idempotent in ${\rm h}\calC$ need not lift to an idempotent in $\calC$ in general.

\begin{theorem}
\label{Thm2-6}
Suppose that the semi-weak Vop\v{e}nka principle holds.
If $\calC$ is a presentable $\infty$-category and $\calS$ is a class of morphisms in~$\calC$, then $\calS^\perp$ is reflective in~$\calC$.
\end{theorem}

\begin{proof}
According to \cite[A.3.7.6]{Lurie}, every presentable $\infty$-category $\calC$ is equivalent to $\calMo$ for some combinatorial simplicial model category~$\calM$. 
Moreover, $\calMo$ is complete and cocomplete, with limits and colimits inherited from homotopy limits and homotopy colimits in~$\calM$; see~\cite[4.2.4.8]{Lurie}.

We keep denoting by $\calS$ the corresponding class of morphisms in~$\calMo$. Since $\calS^\perp(\calM)$ is closed under products in~$\calM$, it follows from Proposition~\ref{Prop2-1} that, given $X\in\calMo$, there is a weak reflection $X\to X^*$ onto $\calS^{\perp}(\calM)$. Then $RQX^*$ is in $\calS^\perp(\calMo)$ and, since $X$ is cofibrant, there is a morphism $X\to RQX^*$ which is a weak reflection of $X$ onto $\calS^{\perp}(\calMo)$ within~$\calMo$.

Now a weak reflection can be turned into a reflection with the same argument as in Proposition~\ref{Prop2-2}, using the fact that $\calMo$ is idempotent complete and has equalizers.
Specifically, if $r\colon X\to D$ is a weak reflection of $X$ onto $\calS^\perp(\calMo)$, let $u\colon E\to D$ be an equalizer of $\prod_k\alpha_k$ and $\prod_k\beta_k$ where $\{(\alpha_k,\beta_k)\}$ are all pairs of morphisms $D\to D$ such that $\alpha_k\circ f\simeq\beta_k\circ f$. Then there is a morphism $s\colon X\to E$ with $u\circ s\simeq r$ and a morphism $t\colon D\to E$ such that $t\circ r\simeq s$. It follows that $t\circ u$ is idempotent and hence, if $t'\colon E\to Z$ is a splitting of $t\circ u$, then $t'\circ s$ is a reflection from $X$ onto $\calS^{\perp}(\calMo)$.
\end{proof}

This result is also contained in~\cite{LoMonaco}, where the following alternative argument was used. If there is a weak reflection of $X$ onto $\calS^\perp(\calMo)$, then the coslice $\infty$-category $\calS^{\perp}(\calMo)_{X/}$ has a weakly initial object.
Since $\calMo$ is complete and $\calS^\perp(\calMo)$ is closed under limits, $\calS^\perp(\calMo)_{X/}$ is also complete. Since it has a weakly initial object, it also has an initial object, as shown in~\cite[Proposition~2.3.2]{Raptis}. This initial object is precisely a reflection of $X$ onto $\calS^\perp(\calMo)$; cf.\;\cite[5.2.7.7]{Lurie}.

\section{Reverse implications}
\label{reverse_implications}

In this section, we address the converse of Theorems \ref{Thm2-4} and~\ref{Thm2-5} using results by Prze\'zdziecki, who proved that, if the weak Vop\v{e}nka principle is false, then there exist non-reflective orthogonality classes of groups \cite{Prz1} and non-reflective orthogonality classes of abelian groups \cite{Prz2}.

In the argument that follows, we use, as in~\cite{Prz1}, the fact that sending every group $G$ to an Eilenberg--Mac Lane space $K(G,1)$ is a full embedding of the category of groups into the homotopy category of pointed simplicial sets, since for all groups $G$ and $H$ there is a natural bijective correspondence between the set of pointed homotopy classes of maps $[K(G,1),K(H,1)]$ and the set of group homomorphisms ${\rm Hom}(G,H)$. 

Moreover, this full embedding preserves and reflects orthogonality, since a group homomorphism $\varphi\colon P\to Q$ is orthogonal to a group $G$ if and only if the map $K(P,1)\to K(Q,1)$ induced by $\varphi$ is orthogonal to a $K(G,1)$, since the space ${\rm map}_*(K(P,1),K(G,1))$ is discrete and its set of connected components is in bijective correspondence with ${\rm Hom}(P,G)$, and similarly with~$Q$.

\begin{theorem}
\label{Thm3-1}
The statement that $\calS$-localization exists in the homotopy category of pointed simplicial sets for every class of maps $\calS$ is equivalent to the weak Vop\v{e}nka principle.
\end{theorem}

\begin{proof}
In Theorem~\ref{Thm2-5}, we have shown that the semi-weak Vop\v{e}nka principle implies the existence of arbitrary $\calS$-localizations of pointed simplicial sets. Since SWVP is equivalent to WVP by~\cite{Wilson1}, it suffices to prove that, under the negation of WVP, there is a class of maps $\calS$ in $\sSet_*$ for which $\calS$-localization does not exist.

Thus suppose that WVP is false. Then, according to~\cite[Proposition~8.7]{Prz1}, there exists a non-reflective orthogonality class $\calG$ of groups. The fact that $\calG$ is an orthogonality class implies that $\calG=\calG^{\perp\perp}$.

Let $\calK$ be the class of Eilenberg--Mac\,Lane spaces $K(G,1)$ with $G\in\calG$. We prove that the class $\calK^{\perp\perp}$ is not reflective in $\Ho(\sSet_*)$. Suppose the contrary, and let $L$ be a reflector. We first show that every \emph{connected} simplicial set in $\calK^{\perp\perp}$ is in~$\calK$. Note that the map $S^2\to *$ is in $\calK^{\perp}$, since ${\rm map}_*(S^2,K(G,1))=\Omega^2K(G,1)$ is contractible for every group~$G$. If $X$ is a connected simplicial set in $\calK^{\perp\perp}$, then $X$ is orthogonal to $S^2\to *$ and hence ${\rm map}_*(S^2,X)$ is contractible. This implies that $\pi_n(X)=0$ for $n\ge 2$, so $X$ is indeed an Eilenberg--Mac Lane space. There remains to show that $\pi_1(X)\in\calG$, which is equivalent to the statement that $\pi_1(X)\in\calG^{\perp\perp}$. Let $\varphi\colon P\to Q$ be any homomorphism in~$\calG^\perp$. Then the induced map $K(P,1)\to K(Q,1)$ is in $\calK^\perp$. Since $X\in\calK^{\perp\perp}$, the map $K(P,1)\to K(Q,1)$ is orthogonal to $X$, and this implies that $\pi_1(X)$ is orthogonal to~$\varphi$, as needed. 

Let $G$ be any group, and let $\ell\colon K(G,1)\to LK(G,1)$ be its localization onto~$\calK^{\perp\perp}$. Here $LK(G,1)$ is connected since every localization of a connected space is connected~\cite{Tai}.
Let us consider the induced group homomorphism $\ell_*\colon G\to H$ where $H=\pi_1(LK(G,1))$. We have that $LK(G,1)=K(H,1)$ with $H\in\calG$, since we have shown that every connected space in $\calK^{\perp\perp}$ is in~$\calK$.
If $J$ is any group in~$\calG$, then the corresponding $K(J,1)$ is in $\calK$ and hence it is orthogonal to~$\ell$. This means precisely that $J$ is orthogonal to $\ell_*$ and therefore $\ell_*$ is a reflection of $G$ onto~$\calG$. Hence the class $\calG$ is reflective, which is a contradiction. 
\end{proof}

\begin{corollary}
\label{Cor3-2}
The statement that $\calS$-localization exists in every presentable $\infty$-category for every class of morphisms $\calS$ is equivalent to the weak Vop\v{e}nka principle.
\end{corollary}

\begin{proof}
One implication has been shown in Theorem~\ref{Thm2-6} by means of SWVP. Conversely, if WVP is false, then, by Theorem~\ref{Thm3-1}, there is a class $\calS$ of maps in $\sSet_*$ for which $\calS^\perp$ is not reflective in $\Ho(\sSet_*)$. Therefore there is also a class of maps between Kan complexes which is not reflective in the $\infty$-categorical sense.
\end{proof}

The stable analogue is similar. In the next result, we use the fact that sending every abelian group $A$ to an Eilenberg--Mac Lane spectrum $HA$ with a single nonzero homotopy group isomorphic to $A$ in dimension $0$ is a full embedding of the category of abelian groups into the homotopy category of spectra. Indeed, the function spectrum $F(HA,HB)$ has two nonzero homotopy groups in general, namely 
\begin{align*}
\pi_0(F(HA,HB)) & \cong{\rm Hom}(A,B) \\[0.1cm] 
\pi_{-1}(F(HA,HB)) & \cong[HA,\Sigma HB]\cong{\rm Ext}(A,B).
\end{align*}
Therefore the connective cover $F^c(HA,HB)$ is an Eilenberg--Mac Lane spectrum whose $\pi_0$ is isomorphic to ${\rm Hom}(A,B)$, and this implies, as in the case of groups, that the full embedding preserves and reflects orthogonality. Thus, a homomorphism of abelian groups $\varphi\colon A\to B$ is orthogonal to an abelian group $C$ if and only if the induced map $HA\to HB$ is orthogonal to~$HC$.

\begin{theorem}
\label{Thm3-3}
The statement that $\calS$-localization exists in the homotopy category of spectra for every class of maps $\calS$ is equivalent to the weak Vop\v{e}nka principle.
\end{theorem}

\begin{proof}
One implication is given by Theorem~\ref{Thm2-4}. To prove the converse, suppose that WVP does not hold. Then there is a non-reflective orthogonality class $\calA$ of abelian groups by~\cite[Proposition~6.8]{Prz2}. Thus, $\calA=\calA^{\perp\perp}$.

Let $\calH$ be the class of Eilenberg--Mac\,Lane spectra $HA$ with $A\in\calA$, and, towards a contradiction, suppose that the class $\calH^{\perp\perp}$ is associated with a localization~$L$ on the homotopy category of spectra. Similarly as in the proof of Theorem~\ref{Thm3-1},
we first show that every \emph{connective} spectrum $X\in\calH^{\perp\perp}$ is in~$\calH$. If $S$ denotes the sphere spectrum, then the map $\Sigma S\to 0$ is in $\calH^{\perp}$, since $F^c(\Sigma S,HA)=0$ for every~$A$. Since $X$ is in~$\calH^{\perp\perp}$, we have that $X$ is orthogonal to $\Sigma S\to 0$ and hence $F^c(\Sigma S,X)=0$. This implies that $\pi_n(X)=0$ for $n\ge 1$, so $X\simeq HE$ for some abelian group~$E$, since $X$ is connective. There remains to show that $E\in\calA$, that is, $E\in\calA^{\perp\perp}$. 
For this, let $\varphi\colon P\to Q$ be any homomorphism in~$\calA^\perp$. Then the induced map $HP\to HQ$ is in $\calH^\perp$. It follows that the map $HP\to HQ$ is orthogonal to~$X$, and this implies that $E$ is orthogonal to~$\varphi$, as needed. 

Let $A$ be any abelian group, and consider the localization $\ell\colon HA\to LHA$ onto the class $\calH^{\perp\perp}$ and the induced group homomorphism $\ell_*\colon A\to B$ where $B=\pi_1(LHA)$. 
Since $\ell$ is also a localization of $HA$ with respect to the set~$\{\ell\}$,
it follows from \cite[Theorem~5.6]{CG} that
$LHA\simeq HB\times\Sigma HC$
for some abelian group $C$. 
Therefore, $LHA$ is connective. Since $LHA\in\calH^{\perp\perp}$, we may infer that $LHA\in\calH$, and consequently $C=0$ and $B\in\calA$.

We finally show that $\ell_*\colon A\to B$ is a reflection of $A$ onto $\calA$ and hence the class $\calA$ is reflective, which is a contradiction. If $E$ is any group in~$\calA$, then $HE$ is in $\calH$ and hence it is orthogonal to $\ell\colon HA\to HB$. This implies that $E$ is orthogonal to~$\ell_*$, as we wanted to prove.
\end{proof}

\subsection{Triangulated categories}
\label{triangulated}

It was shown in \cite[Theorem~2.4]{CGR} that, if Vop\v{e}nka's principle holds, then every full subcategory closed under products and fibres of the homotopy category of a stable locally presentable model category is reflective. Here we improve this result as follows.

\begin{corollary}
\label{Cor3-4}
The statement that every full subcategory closed under products and fibres of the homotopy category of a stable locally presentable model category is reflective is equivalent to the weak Vop\v{e}nka principle.
\end{corollary}

\begin{proof}
If $\calM$ is a stable model category, then the homotopy category $\Ho(\calM)$ is triangulated. Moreover, since every model category is cocomplete, idempotents split in $\Ho(\calM)$. 

Suppose that $\calM$ is locally presentable and let $\calD$ be a full subcategory of $\Ho(\calM)$ closed under products and fibres. Then we may infer that $\calD$ is also closed under retracts, as in~\cite[Lemma~1.4.9]{HPS}.
Assuming WVP, 
Proposition~\ref{Prop2-1} implies that $\calD$ is weakly reflective in $\calM$ and hence also in~$\Ho(\calM)$. Since a homotopy equalizer of two maps $f$ and $g$ is a fibre of the difference $f-g$, it follows from Proposition~\ref{Prop2-2} that $\calD$ is in fact reflective. 

To prove the converse, suppose that WVP does not hold, and hence there exists a non-reflective orthogonality class $\calA$ of abelian groups. If $\calH$ denotes the class of Eilenberg--Mac Lane spectra $HA$ with $A\in\calA$, then $\calH^{\perp\perp}$ is closed under fibres. Therefore a localization $\ell\colon HA\to LHA$ onto $\calH^{\perp\perp}$ exists for every abelian group $A$ and this implies that $\calA$ is reflective as in the proof of Theorem~\ref{Thm3-3}, which is a contradiction.
\end{proof}

Recall that a full subcategory of a triangulated category is \emph{colocalizing} if it is closed under products, fibres, cofibres, and extensions.
We do not know if the statement that every colocalizing subcategory of the homotopy category of spectra is reflective implies the weak Vop\v{e}nka principle (or any other large-cardinal principle). The class $\calH^{\perp\perp}$ used in the proof of  Corollary~\ref{Cor3-4} is not closed under suspensions, hence not colocalizing.

\section{Cohomological localizations}
\label{cohomological-localizations}

The following result is a substantial improvement over the state of the art regarding the existence of cohomological localizations, which is an open problem in~ZFC. It was proved in \cite[Theorem~9.5]{BCMR} that cohomological localizations exist if a proper class of supercompact cardinals exists.

\begin{corollary} 
\label{Cor4-1}
Cohomological localizations of simplicial sets or spectra exist if the weak Vop\v{e}nka principle holds.
\end{corollary}

\begin{proof}
For a generalized cohomology theory $E^*$ defined on simplicial sets or spectra, let $\calS$ be the class of $E^*$-equivalences, that is, maps $X\to Y$ such that the induced homomorphisms $E^n(Y)\to E^n(X)$ are isomorphisms for all~$n\in\ZZ$. Then the reflectivity of $\calS^\perp$ follows from Theorem~\ref{Thm2-4} in the case of spectra and from Theorem~\ref{Thm2-5} in the case of simplicial sets. 
\end{proof}

The consistency strength of the large-cardinal assumption in Corollary~\ref{Cor4-1} can be lowered further as follows. The definition of \emph{$\Pi_n$-strong cardinals} for $n\ge 1$ is given in \cite[Definition~5.1]{BW}, and they form a hierarchy of strictly increasing strength. Moreover, the weak Vop\v{e}nka principle is 
equivalent to the claim that a proper class of $\Pi_n$-strong cardinals exists for all~$n$. 

\begin{theorem}
\label{Thm4-2}
The existence of a proper class of $\Pi_3$-strong cardinals implies the existence of cohomological localizations of simplicial sets or spectra.
\end{theorem}

\begin{proof}
In this proof, for brevity, we call ``space'' a simplicial set or a spectrum.
For the validity of the proof of Proposition~\ref{Prop2-1}, it is sufficient that the semi-weak Vop\v{e}nka principle holds for the sequence of maps 
\begin{equation}
\label{complexity}
\langle r_{i_s}\colon X\to X_{i_s}\mid s\in{\rm Ord}\rangle
\end{equation}
considered in that proof, where $X$ is any space. Let us estimate the L\'evy complexity of a definition of~\eqref{complexity}.
Recall from \cite[\S\,9]{BCMR} that the categories of simplicial sets and Bousfield--Friedlander spectra are~$\Delta_0$ ---that is, they can be defined without unbounded quantifiers---, since they are categories of $\omega$-sorted operational structures. It is also shown in \cite[\S\,9]{BCMR} that the class of $\Omega$-spectra is $\Delta_1$, i.e., it can be defined with a $\Sigma_1$ formula or, alternatively, with a $\Pi_1$ formula. For a fixed $\Omega$-spectrum~$E$, according to \cite[Theorem~9.3]{BCMR}, the class of $E^*$\nobreakdash-equiv\-alences is $\Delta_2$ definable with $E$ is a parameter. 

Given a space~$X$, we need to define a class function $r$ from $\rm Ord$ to objects in $(X\downarrow\sSet)$, by starting with $X\to *$ for $i=0$ and letting $r_i\colon X\to X_i$, for every ordinal~$i\ge 1$, be the product of all maps from $X$ to $E^*$-local spaces whose rank is smaller than~$i$. 
This definition can be formalized with a $\Pi_2$ formula using the following argument, which we owe to Joan Bagaria. 

For $n\ge 0$, let $C^{(n)}$ be the class of ordinals $\theta$ such that $V_\theta$ 
is $\Sigma_n$\nobreakdash-correct in~$V$, that is, $V_\theta$ is a $\Sigma_n$\nobreakdash-elem\-entary substructure of~$V$. As explained in~\cite{Bagaria}, members of $C^{(n)}$ are uncountable limit cardinals for $n>0$, and the statement that $\theta\in C^{(n)}$ is $\Pi_n$ expressible. Moreover, if $\theta\in C^{(n)}$ and $\varphi$ is a $\Sigma_{n+1}$ formula with parameters in $V_\theta$ that holds in~$V_\theta$, then $\varphi$ holds in~$V$.
Hence, if $\theta\in C^{(1)}$ and $E\in V_\theta$, then, if a space $X\in V_\theta$ is orthogonal to all $E^*$-equivalences in $V_\theta$ then $X$ is orthogonal to all $E^*$-equivalences in~$V$, since the class of $E^*$\nobreakdash-equiv\-alences is $\Sigma_2$ definable. 
Therefore, defining $r_i\colon X\to X_i$ amounts to stating that, for all cardinals~$\theta$, if $\theta\in C^{(1)}$ and $\theta>\max\{i,\,{\rm rank}(X),{\rm rank}(E)\}$, then $r_i$ is the product in $V_\theta$ of all maps from $X$ to $E^*$-local spaces in $V_\theta$ whose rank is smaller than~$i$. The complexity of this statement is~$\Pi_2$, because the definition of $r_i$ within $V_\theta$ involves quantifiers bounded by~$V_\theta$.

Next, let $i_0=0$ and recursively define a class function $i\colon{\rm Ord}\to{\rm Ord}$ by picking, for a successor ordinal $s\ge 1$, the smallest ordinal $i_s>i_{s-1}$ such that there is~no map $q\colon X_{i_{s-1}}\to X_{i_s}$ with $r_{i_s}=q\circ r_{i_{s-1}}$. If $s$ is a limit ordinal, then $i_s$ is defined as the smallest ordinal bigger than $i_t$ for all $t<s$ such that there~is no map $q\colon X_{i_t}\to X_{i_s}$ with $r_{i_s}=q\circ r_{i_t}$ for any $t<s$.
To formalize this definition, we need to pick a cardinal $\theta\in C^{(3)}$, since the definition of $r_{i_t}$ for $t<s$ is $\Pi_2$ expressible and we need one more layer of unbounded quantification in order to impose minimality of~$i_s$. Due to the fact that our definition of $i$ is recursive, one additional existential quantifier is required. Hence our complexity assessment for \eqref{complexity} results in~$\Sigma_4$.

The sequence \eqref{complexity} that we have just defined exists if we assume that the class of $E^*$-local spaces is not weakly reflective.
By~\cite[Theorem~5.13]{BW}, if a proper class of $\Pi_3$-strong cardinals exists, then SWVP for $\Sigma_4$ classes with parameters holds. In this case, the existence of the sequence \eqref{complexity} causes a contradiction.
This contradiction tells us that the class of $E^*$-local spaces is weakly reflective, and then Proposition~\ref{Prop2-2} implies that it is reflective. \end{proof}

\section{Examples and counterexamples}
\label{examples}

\begin{example}
\label{ex1} 
This example shows that the statement that every full subcategory closed under products in a locally presentable category is weakly reflective ---which has been proved under SWVP in Proposition~\ref{Prop2-1}--- 
cannot be proved in ZFC.
In \cite[Corollary~2.7]{CGR}, a class of spectra closed under products and retracts but not weakly reflective was exhibited assuming the nonexistence of measurable cardinals. That class consists of Eilenberg--Mac Lane spectra $HA$ where $A$ belongs to the closure of the class of groups $\ZZ^\kappa/\ZZ^{<\kappa}$ under products and retracts, where $\kappa$ runs over all cardinals and $\ZZ^\kappa$ denotes a product of copies of $\ZZ$ indexed by $\kappa$ while $\ZZ^{<\kappa}$ is the subgroup of those sequences whose support has cardinality smaller than~$\kappa$.
\end{example}

\begin{example}
\label{ex2} 
It is not true that every full subcategory closed under products and retracts in $\Ho(\sSet_*)$ is reflective, not even assuming large-cardinal principles. To illustrate this fact, we recall that the class $\calD$ of $1$-connected simplicial sets is closed under products and retracts but it is not reflective. The following argument is due to Mislin \cite[A.1.3]{Farjoun}. Suppose that a map $\ell\colon \RR P^2\to X$ is a reflection in $\Ho(\sSet_*)$ onto~$\calD$, where $\RR P^2$ denotes the real projective plane. Then $\ell$ induces an isomorphism
\[
[X,K(\ZZ,2)]\cong [\RR P^2,K(\ZZ,2)].
\]
However, $[\RR P^2,K(\ZZ,2)]\cong H^2(\RR P^2;\ZZ)\cong\ZZ/2$ while 
\[
[X,K(\ZZ,2)]\cong H^2(X;\ZZ)\cong{\rm Hom}(H_2(X;\ZZ),\ZZ)
\] 
is torsion-free for every $1$-connected space~$X$. This does not contradict Theorem~\ref{Thm2-5} because $\calD$ is not of the form $\calS^\perp$ for any class of maps $\calS$, since the homotopy fibre of a map between $1$-connected spaces need not be $1$-connected, e.g., $S^1\to S^3\to S^2$, and therefore $\calD$ is not closed under homotopy limits.

It is interesting to note that, nevertheless, the class $\calD$ of $1$-connected spaces is indeed weakly reflective. A~weak reflection can be described explicitly for topological spaces as follows. Given any space~$X$, define $X\hookrightarrow X(0)$ by choosing a point in each path-connected component of $X$ and attaching an edge between every pair of such points. Next, define $X(0)\hookrightarrow X(1)$ by choosing a presentation of the fundamental group $\pi_1(X(0))$ and attaching a $2$-cell to $X(0)$ for each generator in the chosen presentation. Then the space $X(1)$ is $1$-connected and every continuous map $X\to Y$ where $Y$ is $1$-connected extends to~$X(1)$. Similarly, the class of $n$-connected spaces is weakly reflective for all~$n$.

As another example, the class $\calQ$ of spaces whose fundamental group is uniquely radicable and whose higher homotopy groups are $\QQ$-vector spaces for every choice of a basepoint is closed under products and retracts but it is not reflective in $\Ho(\sSet_*)$, as shown in~\cite{PAMS}.
The reason is that if $\ell\colon S^1\to X$ were a reflection onto~$\calQ$, then it would follow that $\pi_1(X)\cong\QQ$ and 
\[
H^2(X;\QQ[\QQ])\cong H^2(S^1;\QQ[\QQ]),
\]
where $\QQ[\QQ]$ denotes the group ring of $\QQ$ with rational coefficients, and cohomology with twisted coefficients is meant. However, $H^2(S^1;\QQ[\QQ])=0$, while, as proved in \cite[Proposition~2.1]{PAMS}, $H^2(X;\QQ[\QQ])\ne 0$, which is a contradiction.
\end{example}

\begin{example}
\label{ex3}
There are reflections onto full subcategories closed under products and retracts in $\Ho(\calM)$ for a simplicial model category $\calM$ that are not $\calS$-localizations for any class of morphisms~$\calS$, and, moreover, cannot be lifted to coaugmented functors on~$\calM$.
Our main example involves the class $\calD$ of connective spectra whose homotopy groups are $\QQ$\nobreakdash-vector spaces. A~reflection onto $\calD$ in $\Ho(\Spectra)$ can be given explicitly as follows. For an arbitrary spectrum $X$, let $X\wedge H\QQ$ be its rationalization. Since $X\wedge H\QQ$ splits as a wedge $\bigvee_{k\in\ZZ}\, \Sigma^kH(\pi_k(X)\otimes\QQ)$, we can retract it into 
\[
LX=\bigvee_{k\ge 0}\, \Sigma^kH(\pi_k(X)\otimes\QQ),
\]
or into any other segment. The composite $\ell\colon X\to LX$ is a reflection onto~$\calD$, since every map $X\to Y$ where $Y\in\calD$ factors uniquely through $X\wedge H\QQ$ up to homotopy, and $[\Sigma^kHA,Y]=0$ if $k<0$, for all~$A$, since $Y$ splits and ${\rm Ext}(A,\QQ)=0$ for all~$A$.
However, the class $\calD$ is not closed under fibres and therefore $L$ is not an $\calS$-localization for any class of maps~$\calS$.

In this example, $L$ can be lifted to an endofunctor on $\Spectra$, namely the composite of $(-)\wedge H\QQ$ with passage to the connective cover. However, while the first functor is coaugmented, the second functor is augmented, that is, there is a zig-zag of natural transformations
\[
X\longrightarrow X\wedge H\QQ \longleftarrow (X\wedge H\QQ)^c.
\]
The second arrow can be reversed in $\Ho(\Spectra)$, but not in $\Spectra$. Indeed,
it follows from \cite[Theorem~2.2]{CCh} that there does not exist any natural transformation ${\rm Id}\to L$ in $\Spectra$ lifting the unit $\ell$ of the reflector $L$ on $\Ho(\Spectra)$.
\end{example}

\section{Bousfield localizations}
\label{functoriality}

In this section we discuss to what extent an $\calS$-localization on a combinatorial model category $\calM$ obtained using WVP can be enhanced to a left Bousfield localization. The main difficulty is that WVP yields homotopical localizations that are not necessarily functorial on~$\calM$.

Left Bousfield localization of a model category $\calM$ with respect to a class of morphisms $\calS$ is a model structure $\calM_{\cal S}$ on the same underlying category as $\calM$ together with a left Quillen functor $\calM\to\calM_{\calS}$ which is initial among left Quillen functors that send morphisms in $\calS$ to weak equivalences. Such a model category is known to exist if $\calM$ is combinatorial and left proper and $\calS$ is a (small) set. If a left Bousfield localization $\calM_{\calS}$ exists, then a fibrant replacement on $\calM_{\calS}$  yields an $\calS$-localization on~$\calM$.

A \emph{semi-model category} is defined with the same axioms as a model category, except that the lifting axiom and the factorization axiom hold only for morphisms with fibrant codomain (in \emph{right} semi-model categories) or instead with cofibrant domain (in \emph{left} semi-model categories). It was shown in \cite{BWh} that if the assumption that $\calM$ be left proper is omitted, then a Bousfield localization $\calM_\calS$ for a (small) set of morphisms $\calS$ still exists as a left semi-model category. An example where left properness fails and left Bousfield localization does not exist as a model category is given in \cite[Example~3.48]{Voevodsky}.

\begin{proposition}
\label{Prop6-1}
Let $\calS$ be a class of morphisms in a 
combinatorial simplicial 
model category $\calM$ in which all objects are cofibrant. If the weak Vop\v{e}nka principle holds, then a left Bousfield localization $\calM_\calS$ exists as a right semi-model category.
\end{proposition}

\begin{proof}
Consider the $\infty$-category $\calC=\calM^\circ$ of fibrant objects in~$\calM$. Since $\calM$ is combinatorial, $\calC$ is presentable, according to \cite[A.3.7.6]{Lurie}. Hence, Theorem~\ref{Thm2-6} yields an $\infty$-categorical reflection $E\colon\calC\to\calC$ onto $(R\calS)^\perp$, where $R$ is a fibrant replacement functor on~$\calM$ ---we omit cofibrant replacement since, by assumption, all objects in $\calM$ are cofibrant. By Corollary~\ref{Cor1-3}, the composite $X\to RX\to ERX$ is an $\calS$-localization of $X$ for every $X$ in~$\calM$.

Let us factor this composite, for every~$X$, into a cofibration $\ell_X\colon X\to LX$ followed by a trivial fibration $LX\to ERX$. Hence $\ell_X$ is also an $\calS$-localization of~$X$.
Now, given any morphism $f\colon X\to Y$ in~$\calM$, the diagram
\begin{equation}
\label{square}
\xymatrix@C=4pc@R=2pc{
  X \ar[r]^-{\ell_X} \ar[d]_{f} & LX \\
  Y \ar[r]^-{\ell_Y} & LY
}
\end{equation}
can be closed strictly. Indeed, since $LY$ is fibrant and there is a morphism $LX\to LY$ making \eqref{square} homotopy commutative, we can replace this morphism with a homotopic one, which we denote by $Lf\colon LX\to LY$,
such that $Lf\circ\ell_X=\ell_Y\circ f$; see \cite[A.2.3.1]{Lurie} for details.

In this situation, even though $L$ need not be an endofunctor of~$\calM$, condition A.2 from \cite{BCh} is fulfilled. Conditions A.3 and A.4 from \cite{BCh} are also fulfilled, since the class of $L$-equivalences is equal to $\calS^{\perp\perp}$, and condition A.5 is checked with a similar argument as in \cite[Proposition~6.6]{BCh}. 

As pointed out in \cite[Proposition~3.13]{Carmona}, the fact that $(L,\ell)$ satisfies conditions A.2--A.5 from \cite{BCh} ensures the existence of a right semi-model structure $\calM_\calS$ as in \cite[Theorem~3.5]{Carmona}. \end{proof}

In the semi-model structure given by Proposition~\ref{Prop6-1}, factorizations are not necessarily functorial, due to the fact that there need not be a fibrant replacement functor on $\calM_\calS$ lifting~$L$.

In favorable cases, the semi-model category structure on $\calM_\calS$ given by Proposition~\ref{Prop6-1} can be further enhanced to a model category structure, as in the following situation. Given a class $\calW$ of pointed simplicial sets, a Kan complex $X$ is called \emph{$\calW$-null} if the map from $X$ to $\map(W,X)$ induced by $W\to *$ is a weak equivalence for all $W\in\calW$. If $X$ is connected, this is equivalent to imposing that $\map_*(W,X)$ be contractible for all $W\in\calW$. Nullification with respect to $\calW$ is a special case of an $\calS$-localization, by letting $\calS$ be the class of maps $W\to *$ for all $W\in\calW$. The main properties of nullification functors with respect to sets carry over to nullifications with respect to classes, provided that these exist. Specifically, in the proof of the next result we use \cite[Corollary~4.8]{BousfieldJAMS} and \cite[Corollary~D.3]{Farjoun}.

\begin{proposition}
\label{Prop6-2}
Let $\calS$ be a class of maps of pointed simplicial sets whose codomain is a one-point space. If the weak Vop\v{e}nka principle holds, then a left Bousfield localization $\calM_\calS$ exists as a model category.
\end{proposition}

\begin{proof}
Under WVP, an $\calS$-localization $L$ exists on $\Ho(\sSet_*)$. Furthermore, as in the proof of Proposition~\ref{Prop6-1}, we can assume that $L$ satisfies conditions A.2--A.5 from \cite{BCh}. 
Although condition A.6 is hard to verify, a more useful alternative condition (A3) was  stated in \cite[\S\,9.2]{BousfieldTAMS}, by imposing that pullbacks of $L$-equivalences along fibrations between $L$-local fibrant objects are $L$\nobreakdash-equiv\-alences. 
As we next show, this condition is satisfied if $L$ is a nullification of simplicial sets ---in fact, by \cite[Remark~9.11]{BousfieldTAMS}, if (A3) holds then $L$ is necessarily a nullification. Thus, suppose given a pull-back square
\begin{equation}
\label{A3}
\xymatrix@C=3pc@R=3pc{
  A \ar[r] \ar[d]_{h^*} & B\ar[d]^{h} \\
  X \ar[r]^-{f} & Y
}
\end{equation}
in which $h$ is an $L$-equivalence and $f$ is a fibration between $L$-local Kan complexes. We can assume that \eqref{A3} is a homotopy pull-back square. Since $Y$ is $L$-local and $L$ is a nullification, the fibre of $h$ is sent by $L$ to a contractible space. Consequently, the fibre of $h^*$ is also sent by $L$ to a contractible space and this implies that $h^*$ is an $L$-equivalence. Since (A3) holds, the existence of a left Bousfield localization $\calM_\calS$ follows as in \cite[Theorem~A.8]{BCh}.
\end{proof}

According to \cite[Theorem~A.8]{BCh}, the model structure $\calM_\calS$ provided by Proposition~\ref{Prop6-2} is right proper. An example is Quillen's plus-construction, which is a nullification with respect to the class of all $H_*$-acyclic spaces for ordinary homology. However, left Bousfield localization of $\sSet_*$ at the class of $H_*$-equivalences is not right proper, since the fibre of an $H_*$-equivalence need not be $H_*$-acyclic, not even if the codomain is an $H_*$-local space.

Proposition~\ref{Prop6-2} also holds for spectra, with the same proof (replacing the one-point space with the zero spectrum), if we choose the injective stable model structure on symmetric spectra over simplicial sets~\cite{HSS}, which is combinatorial and simplicial, and every spectrum in it is cofibrant.

As proved in~\cite{Dugger1}, every combinatorial model category is Quillen equivalent to a simplicial one in which every object is cofibrant. Hence, the assumptions made in Proposition~\ref{Prop6-1} that $\calM$ is simplicial and that all objects of $\calM$ are cofibrant can be removed as follows. 

\begin{proposition}
\label{Prop6-3}
Let $\calS$ be a class of morphisms in a combinatorial model category $\calM$. If the weak Vop\v{e}nka principle holds, then $\calS$\nobreakdash-local\-ization exists on $\Ho(\calM)$ and can be enhanced to a left Bousfield localization, as a right semi-model category, on a combinatorial simplicial model category Quillen equivalent to~$\calM$.
\end{proposition}

\begin{proof}
As shown in \cite[Corollary~1.2]{Dugger1}, for every combinatorial model category $\calM$ there is a zig-zag of Quillen equivalences
\begin{equation}
\label{dugger}
\calM'\longleftarrow \calM''\longrightarrow \calM,
\end{equation}
where the arrow indicates the left adjoint, in which $\calM''$ is a left Bousfield localization of a model category of simplicial presheaves and $\calM'$ is the corresponding localization of the same category endowed with the Heller model structure, in which all objects are cofibrant. Both $\calM'$ and $\calM''$ are combinatorial and simplicial.

By \cite[Theorem~2.1]{Mazel-gee}, the $\infty$\nobreakdash-categories $\mathbf{R} L^H\calM$ and $\mathbf{R} L^H\calM'$ are equivalent, where $\mathbf{R}$ denotes fibrant replacement in the Bergner model structure,
and $\mathbf{R} L^H\calM'$ is equivalent to $(\calM')^\circ$ because $\calM'$ is simplicial \cite[\S\;4]{DK2}. Since $\calM'$ is combinatorial,
$(\calM')^\circ$ is presentable, and hence, by Theorem~\ref{Thm2-6}, WVP implies that $D\calS$-localization exists on $(\calM')^\circ$, where $D$ is a composite of derived Quillen functors from $\calM$ to $\calM'$ in~\eqref{dugger}. Since all objects of $\calM'$ are cofibrant, $D\calS$\nobreakdash-localization on $(\calM')^\circ$ can be enhanced to a left Bousfield localization of $\calM'$ as a right semi-model structure as in Proposition~\ref{Prop6-1}.
\end{proof}

As observed in \cite{CCh,LoMonaco,RT}, if one uses VP instead of~WVP, then left Bousfield localizations of left proper combinatorial model categories at classes of morphisms exist without need of further assumptions.

\vskip 0.3cm

\noindent
Departament de Matem\`atiques i Inform\`atica, Universitat de Barcelona (UB), Gran Via de les Corts Catalanes 585, 08007 Barcelona

\vskip 0.2cm

\noindent
carles.casacuberta@ub.edu

\noindent
javier.gutierrez@ub.edu


\begin{thebibliography}{99}

\bibitem{AR1} J. Ad\'{a}mek and J. Rosick\'{y}, On injectivity in locally presentable categories, {\it Trans. Amer. Math. Soc.} {\bf 336} (1993), no.\;2, 785--804.
       
\bibitem{AR} J. Ad\'{a}mek and J. Rosick\'{y}, \textit{Locally Presentable and Accessible Categories}, London Math. Soc. Lecture Note Ser., vol.\;189, Cambridge University Press, Cambridge, 1994.

\bibitem{Bagaria} J. Bagaria,
$C^{(n)}$-cardinals, \textit{Arch. Math. Logic} \textbf{51} (2012), 213--240.

\bibitem{BCMR} J. Bagaria, C. Casacuberta, A.R.D. Mathias and J. Rosick\'{y}, Definable orthogonality classes in accessible categories are small, {\it J. Eur. Math. Soc.} {\bf 17} (2015), no.\;3, 549--589.

\bibitem{BW} J. Bagaria and T. M. Wilson, The weak Vop\v{e}nka principle for definable classes of structures, {\it J. Symb. Log.} {\bf 88} (2023), no.\;1, 145--168.

\bibitem{BWh} M. A. Batanin and D. White, Left Bousfield localization without left properness, {\it J. Pure Appl. Algebra} {\bf 228} (2024), no.\;6, 107570.

\bibitem{Bergner} J. Bergner, A model category structure on the category of simplicial categories, {\it Trans. Amer. Math. Soc.} {\bf 359} (2007), 2043--2058.

\bibitem{BCh} G. Biedermann and B. Chorny, Duality and small functors, {\it Alg. Geom. Topol.} {\bf 15} (2015), 2609--2657.

\bibitem{Bousfield} A. K. Bousfield, Unstable localization and periodicity, in: Algebraic Topology; New Trends in Localization and Periodicity (Sant Feliu de Gu\'{\i}xols, 1994), Progress in Math., vol.\;136, Birkh\"auser, Basel, 1996, 33--50.

\bibitem{BousfieldJAMS} A. K. Bousfield, Localization and periodicity in unstable homotopy theory, {\it J. Amer. Math. Soc.} {\bf 7} (1994), no.\;4, 831--873.

\bibitem{BousfieldTAMS} A. K. Bousfield, On the telescopic homotopy theory of spaces, {\it Trans. Amer. Math. Soc.} {\bf 353} (2001), no.\;6, 2391--2426.

\bibitem{BF} A. K. Bousfield and E. M. Friedlander, Homotopy theory of {$\Gamma $}-spaces, spectra, and bisimplicial sets, {\it Geometric Applications of Homotopy Theory ({E}vanston, 1977), {II}}, 80--130, {\it Lecture Notes in Math.} vol.\;658. Springer, Berlin, 1978.

\bibitem{Carmona} V. Carmona, When Bousfield localizations and homotopy idempotent functors meet again, {\it Homol. Homotopy Appl.} {\bf 25} (2023), no.\;2, 187--218.

\bibitem{PAMS} C. Casacuberta, On the rationalization of the circle, {\it Proc. Amer. Math. Soc.} {\bf 118} (1993), no.\;3, 995--1000.

\bibitem{CCh} C. Casacuberta and B. Chorny, 
The orthogonal subcategory problem in homotopy theory, in:
An Alpine Anthology of Homotopy Theory (Arolla, 2004), Contemp. Math. vol.\;399, Amer. Math. Soc., Providence (2006), 41--53.

\bibitem{CG} C. Casacuberta and J. J. Guti\'{e}rrez, Homotopical localizations of module spectra, {\it Trans. Amer. Math. Soc.} {\bf 357} (2005), no.\;7, 2753--2770.

\bibitem{CGR} C. Casacuberta, J. J. Guti\'{e}rrez, and J. Rosick\'{y}, Are all localizing subcategories of stable homotopy categories coreflective?, {\it Adv. Math.} {\bf 252}
(2014), 158--184.

\bibitem{CSS} C. Casacuberta, D. Scevenels, and J. H. Smith, Implications of large-cardinal principles in homotopical   localization, {\it Adv. Math.} {\bf 197} (2005), no.\;1, 120--139.

\bibitem{Farjoun}  E. Dror Farjoun, {\it Cellular Spaces, Null Spaces and Homotopy Theory}, Lecture Notes in Math., vol.\;1622, Springer-Verlag, Berlin, Heidelberg, 1996.

\bibitem{Dugger1} D. Dugger, 
Combinatorial model categories have presentations, {\it Adv. Math.} {\bf 164} (2001), 177--201.

\bibitem{Dugger2} D. Dugger, Replacing model categories with simplicial ones, {\it Trans. Amer. Math. Soc.} {\bf 353} (2001), no.\;12, 5003--5027.

\bibitem{DK1}
W. G. Dwyer and D. M. Kan, Calculating simplicial localizations, {\it J. Pure Appl. Algebra} {\bf 18} (1980), 17--35.

\bibitem{DK2}
W. G. Dwyer and D. M. Kan, Function complexes in homotopical algebra, {\it Topology} {\bf 19} (1980), 427--440.

\bibitem{FH} P. Freyd and A. Heller, Splitting homotopy idempotents II, {\it J. Pure Appl. Algebra} {\bf 89} (1993), no.\;1-2, 93--106.

\bibitem{FK} P. J. Freyd and G. M. Kelly, Categories of continuous functors~I, \emph{J. Pure Appl. Algebra} \textbf{2} (1972), 169--191.

\bibitem{GU} P. Gabriel and F. Ulmer, \textit{Local pr\"asentierbare Kategorien},
Lecture Notes in Math., vol.\;221, Springer, Berlin, Heidelberg, 1971.

\bibitem{Hinich} V. Hinich, Dwyer--Kan localization revisited, \textit{Homol. Homotopy Appl.\/} \textbf{18} (2016), no.\;1, 27--48.

\bibitem{Hirschhorn} P. S. Hirschhorn, \textit{Model Categories and their Localisations}, Math. Surveys and Monographs, vol.\;99, Amer. Math. Soc., Providence, 2003.

\bibitem{HPS} M. Hovey, J. H. Palmieri, N. P. Strickland, \textit{Axiomatic Stable Homotopy Theory}, 
Mem. Amer. Math. Soc., vol.\;128, no.\;610, Amer. Math. Soc., Providence, 1997.

\bibitem{HSS} M. Hovey, B. Shipley, and J. H. Smith, Symmetric spectra, {\it J. Amer. Math. Soc.} {\bf 13} (2000), 149--208.

\bibitem{Jech} T. Jech, \emph{Set Theory. The Third Millenium Edition, Revised and Expanded}, Springer Monographs in Math., Springer-Verlag, Berlin, Heidelberg, 2003.

\bibitem{Joyal} A. Joyal, Quasi-categories and Kan complexes, \emph{J. Pure Appl. Algebra} \textbf{175} (2002), 207--222.

\bibitem{Levy} A. L\'evy, \emph{A Hierarchy of Formulas in Set Theory}, Mem. Amer. Math. Soc., vol.\;57, Amer. Math. Soc., Providence, 1965.

\bibitem{LoMonaco} G. Lo Monaco, Vop\v{e}nka's principle in $\infty$-categories, \emph{J. Pure Appl. Algebra} \textbf{228} (2024), 107633, 30 pp.

\bibitem{Lurie} J. Lurie, \emph{Higher Topos Theory}, Annals of Mathematics Studies, vol.\;170, Princeton University
Press, 2009.

\bibitem{Mazel-gee} A. Mazel-Gee, Quillen adjunctions induce adjunctions of quasicategories, {\it New York J. Math.} \textbf{22} (2016), 57--93.

\bibitem{Neeman} A. Neeman, \textit{Triangulated Categories},
Annals of Mathematics Studies, vol.\;148, Princeton University Press, Princeton, 2001.

\bibitem{Raptis}
H. K. Nguyen, G. Raptis, and C. Schrade, Adjoint functor theorems for $\infty$-categories, \textit{J. Lond. Math. Soc.} \textbf{101} (2020), no.\;2, 659--681.

\bibitem{Prz1} A. Prze\'{z}dziecki, An ``almost'' full embedding of the category of graphs into the category of groups, {\it Adv. Math.} {\bf 225} (2010), no.\;4, 1893--1913.

\bibitem{Prz2} A. Prze\'{z}dziecki, An almost full embedding of the category of graphs into the category of abelian groups,
{\it Adv. Math.} {\bf 257} (2014), 527--545.
		
\bibitem{Quillen} D. G. Quillen, {\it Homotopical Algebra},
Lecture Notes in Math. vol.\;43, Springer-Verlag, Berlin, New York, 1967.

\bibitem{RSS} C. Rezk, S. Schwede, and B. Shipley, Simplicial structures on model categories and functors, {\it Amer. J. Math.} {\bf 123} (2001), 551--575.

\bibitem{RT} J. Rosick\'{y} and W. Tholen, Left-determined model categories and universal homotopy theories, {\it Trans. Amer. Math. Soc.} {\bf 355} (2003), no.\;9, 3611--3623.

\bibitem{Tai} J.-Y. Tai,
On $f$-localization functors and connectivity,
in: Stable and Unstable Homotopy, Fields Inst. Commun., vol.\;19, Amer. Math. Soc., Providence, 1998, 285--298.

\bibitem{Voevodsky} V. Voevodsky, Simplicial radditive functors, {\it J. K-Theory} {\bf 5} (2010), 201--244.

\bibitem{Wilson1} T. M. Wilson, Weak {V}op\v{e}nka's principle does not imply {V}op\v{e}nka's principle, {\it Adv. Math.} {\bf 363} (2020), 106986, 11 pp.

\bibitem{Wilson2} T. M. Wilson, The large cardinal strength of weak Vop\v{e}nka's principle, \emph{J. Math. Logic} {\bf 22} (2022), 2150024, 15 pp.

\end{thebibliography}
\end{document}